\documentclass[11pt]{amsart}

\usepackage{amsmath}
\usepackage[active]{srcltx}
\usepackage{t1enc}
\usepackage[latin2]{inputenc}
\usepackage{verbatim}
\usepackage{amsmath,amsfonts,amssymb,amsthm}
\usepackage[mathcal]{eucal}
\usepackage{enumerate}
\usepackage[centertags]{amsmath}
\usepackage{graphics}

\newtheorem{theorem}{Theorem}

\newtheorem{lemma}{Lemma}

\newtheorem*{FS}{Theorem FS}

\begin{document}
\author{Ushangi Goginava }
\address{U. Goginava, Department of Mathematics, Faculty of Exact and Natural
Sciences, Tbilisi State University, Chavcha\-vadze str. 1, Tbilisi 0128,
Georgia}
\email{zazagoginava@gmail.com}
\title[Strong Summability]{Strong Summability of Two-Dimensional Vilenkin
Fourier series }
\date{}
\maketitle

\begin{abstract}
In this paper we study the exponential uniform strong summability of
two-dimensional Vilenkin-Fourier series. In particular, it is proved that
the two-dimensional Vilenkin-Fourier series of the continuous function $f$
is uniformly strong summable to the function $f$ exponentially in the power $%
1/2$. Moreover, it is proved that this result is best possible.
\end{abstract}

\footnotetext{%
2010 Mathematics Subject Classification 42C10 .
\par
Key words and phrases: Vilenkin function, Best approximation, Strong
Approximation.
\par
The research was supported by Shota Rustaveli National Science Foundation
grant no.DI/9/5-100/13 (Function spaces, weighted inequalities for integral
operators and problems of summability of Fourier series)}

\section{Introduction}

It is known that there exist continuous functions the trigonometric (Walsh)
Fourier series of which do not converge. However, as it was proved by Fejér 
\cite{Fe} in 1905, the arithmetic means of the differences between the
function and its Fourier partial sums converge uniformly to zero. The
problem of strong summation was initiated by Hardy and Littlewood \cite{HL}.
They generalized Fejér's result by showing that the strong means also
converge uniformly to zero for any continous function. The investigation of
the rate of convergence of the strong means was started by Alexits \cite{AK}%
. Many papers have been published which are closely related with strong
approximation and summability. We note that a number of signficant results
are due to Leindler \cite{Le1,Le2,Le3}, Totik \cite{To1,To2,To3}, Gogoladze 
\cite{Go}, Goginava, Gogoladze, Karagulyan \cite{GGKCA}. Leindler has also
published a monograph \cite{Le4}.

The results on strong summation and approximation of trigonometric Fourier
series have been extended for several other orthogonal systems. For
instance, concerning the Walsh system see Schipp \cite{Sch1,Sch2,Sch3},
Fridli, Schipp \cite{FS,FS2}, Fridli \cite{kacz}, Rodin \cite{Ro1},
Goginava, Gogoladze \cite{GGSMH, GGCA}, Gat, Goginava, Karagulyan \cite%
{GGKAM,GGKJMAA}, Goginava, Gogoladze, Karagulyan \cite{GGKCA} and concerning
the Ciselski system see Weisz \cite{We1,We2}. The summability of multiple
Walsh-Fourier series have been investigated in (\cite{GoAM}, \cite{GoSM}, 
\cite{Webook2})

Fridli and Schipp \cite{FS2} proved that the following is true.

\begin{FS}
Let $\Phi $ stand for the trigonometric or the Walsh system, and let $\psi $
be a monotonically increasing function defined on $[0,\infty )$ for which $%
\lim_{u\rightarrow 0+}\psi (u)=0$. Then 
\begin{equation*}
\lim\limits_{n\rightarrow \infty }\frac{1}{n}\sum\limits_{k=1}^{n}\psi
\left( \left\vert S_{k}^{\Phi }f(x)-f(x)\right\vert \right)
=0\,\,\,\,\,\,\,\,\,\,\,\,(f\in C(G_{2}))
\end{equation*}%
if and only if there exists $A>0$ such that $\psi (t)\leq \exp (At)$ $(0\leq
t<\infty )$. Moreover, the convergence is uniform in $x$.
\end{FS}

In this paper we study the exponential uniform strong summability of
two-dimensional Vilenkin-Fourier series. In particular, it is proved that
the two-dimensional Vilenkin-Fourier series of the continuous function $f$
is uniformly strong summable to the function $f$ exponentially in the power $%
1/2$. Moreover, it is proved that this result is best possible.

Let $\mathbb{N}_{+}$ denote the set of positive integers, $\mathbb{N}:=%
\mathbb{N}_{+}\cup \{0\}.$ Let $m:=(m_{0},m_{1},...)$ denote a sequence of
positive integers not less than $2.$ Denote by $Z_{m_{k}}:=\{0,1,...,m_{k}-1%
\}$ the additive group of integers modulo $m_{k}$. Define the group $G_{m}$
as the complete direct product of the groups $Z_{m_{j}},$ with the product
of the discrete topologies of $Z_{m_{j}}$'s. The direct product $\mu $ of
the measures 
\begin{equation*}
\mu _{k}(\{j\}):=\frac{1}{m_{k}}\quad (j\in Z_{m_{k}})
\end{equation*}%
is the Haar measure on $G_{m}$ with $\mu (G_{m})=1.$ If the sequence $m$ is
bounded, then $G_{m}$ is called a bounded Vilenkin group. The elements of $%
G_{m}$ can be represented by sequences $x:=(x_{0},x_{1},...,x_{j},...)$, $%
(x_{j}\in Z_{m_{j}}).$ The group operation $+$ in $G_{m}$ is given by $%
x+y=\left( x_{0}+y_{0}\left( \text{mod}m_{0}\right) ,...,x_{k}+y_{k}\left( 
\text{mod}m_{k}\right) ,...\right) $ , where $x=\left(
x_{0},...,x_{k},...\right) $ and $y=\left( y_{0},...,y_{k},...\right) \in
G_{m}$. The inverse of $+$ will be denoted by $-$.

It is easy to give a base for the neighborhoods of $G_{m}:$ 
\begin{equation*}
I_{0}(x):=G_{m},
\end{equation*}%
\begin{equation*}
I_{n}(x):=\{y\in G_{m}|y_{0}=x_{0},...,y_{n-1}=x_{n-1}\}
\end{equation*}%
for $x\in G_{m},\ n\in {\mathbb{N}}$. Define $I_{n}:=I_{n}(0)$ for $n\in {%
\mathbb{N}}_{+}$. Set $e_{n}:=\left( 0,...,0,1,0,...\right) \in G_{m}$ the $%
n\,$th\thinspace coordinate of which is 1 and the rest are zeros $\left(
n\in \mathbb{N}\right) .$

If we define the so-called generalized number system based on $m$ in the
following way: $M_{0}:=1,M_{k+1}:=m_{k}M_{k}(k\in {\mathbb{N}}),$ then every 
$n\in {\mathbb{N}}$ can be uniquely expressed as $n=\sum\limits_{j=0}^{%
\infty }n_{j}M_{j},$ where $n_{j}\in Z_{m_{j}}\ (j\in {\mathbb{N}}_{+})$ and
only a finite number of $n_{j}$'s differ from zero. We use the following
notation. Let (for $n>0$) $|n|:=\max \{k\in {\mathbb{N}}:n_{k}\ne 0\}$ (that
is, $M_{|n|}\le n<M_{|n|+1}$).

Next, we introduce on $G_{m}$ an orthonormal system which is called the
Vilenkin system. At first define the complex valued functions $%
r_{k}(x):G_{m}\to {\mathbb{C}}$, the generalized Rademacher functions in
this way 
\begin{equation*}
r_{k}(x):=\exp \frac{2\pi \imath x_{k}}{m_{k}}\ (\imath ^{2}=-1,\ x\in
G_{m},\ k\in \mathbb{N}).
\end{equation*}

\noindent Now define the Vilenkin system $\psi := (\psi_n : n\in{\mathbb{N}}%
) $ on $G_{m}$ as follows. 
\begin{equation*}
\psi_{n}(x):=\prod\limits_{k=0}^{\infty}r_{k}^{n_{k}}(x)\quad (n\in\mathbb{N}%
).
\end{equation*}

\noindent Specifically, we call this system the Walsh-Paley one if $m\equiv
2.$

\noindent The Vilenkin system is orthonormal and complete in $L_{1}(G_{m})$.
It is well-known that $\psi _{n}(x)\psi _{n}(y)=\psi _{n}(x+y)$, $\left\vert
\psi _{n}(x)\right\vert =1\left( k,n\in \mathbb{N}\right) ,$ $\psi _{n}(-x)=%
\overline{\psi }_{n}(x)$ \cite{SWS}.

\noindent Now, introduce analogues of the usual definitions of the Fourier
analysis. If $f\in L_{1}(G_{m})$ we can establish the following definitions
in the usual way:

\noindent Fourier coefficients: 
\begin{equation*}
\widehat{f}(k):=\int_{G_{m}}f\overline{\psi }_{k}d\mu \qquad (k\in {\mathbb{N%
}}),
\end{equation*}%
partial sums: 
\begin{equation*}
S_{n}f:=\sum_{k=0}^{n-1}\widehat{f}(k)\psi _{k}\qquad (n\in {\mathbb{N}}%
_{+},\,\,S_{0}f:=0).
\end{equation*}%
Dirichlet kernels: 
\begin{equation*}
D_{n}:=\sum_{k=0}^{n-1}\psi _{k}\qquad (n\in {\mathbb{N}}_{+}).
\end{equation*}

Recall that 
\begin{equation}
D_{M_{n}}\left( x\right) =\left\{ 
\begin{array}{ll}
M_{n}, & \mbox{if }x\in I_{n}, \\ 
0, & \mbox{if }x\in G_{m}\backslash I_{n}.%
\end{array}%
\right. ,  \label{dir}
\end{equation}%
\begin{equation}
D_{n}\left( x\right) =\psi _{n}\left( x\right) \sum\limits_{j=0}^{\infty
}D_{M_{j}}\left( x\right)
\sum\limits_{q=m_{j}-n_{j}}^{m_{j}-1}r_{j}^{q}\left( x\right) ,\text{ \ }%
\left( f\in L_{1}\left( G_{m}\right) ,n\in \mathbb{N}\right) .  \label{dir2}
\end{equation}

It is well known that 
\begin{equation*}
\sigma _{n}\left(f; x\right) =\int\limits_{G_{m}}f\left( t\right) K_{n}\left(
x-t\right) d\mu \left( t\right) .
\end{equation*}

Next, we introduce some notation with respect to the theory of
two-dimensional Vilenkin system. Let us fix $d\geq 1,d\in \mathbb{N}_{+}$.
For Vilenkin group $G_{m}$ let $G_{m}^{d}$ be its Cartesian product $%
G_{m}\times \cdots \times G_{m}$ taken with itself $d$-times. Denote by $%
\mathbf{\mu }$ the product measure $\mu \times \cdots \times \mu $. The
rectangular partial sums of the two-dimensional Vilenkin-Fourier series are
defined as follows:

\begin{equation*}
S_{M,N}(f;x,y):=\sum\limits_{i=0}^{M-1}\sum\limits_{j=0}^{N-1}\widehat{f}%
\left( i,j\right) \psi _{i}\left( x\right) \psi _{j}\left( y\right) ,
\end{equation*}%
where the number 
\begin{equation*}
\widehat{f}\left( i,j\right) =\int\limits_{G_{m}\times G_{m}}f\left(
x,y\right) \overline{\psi }_{i}\left( x\right) \overline{\psi }_{j}\left(
y\right) d\mathbf{\mu }\left( x,y\right) .
\end{equation*}%
is said to be the $\left( i,j\right) $th Vilenkin-Fourier coefficient of the
function\thinspace $f.$

Denote 
\begin{equation*}
S_{n}^{\left( 1\right) }\left( f;x,y\right) :=\sum\limits_{l=0}^{n-1}%
\widehat{f}\left( l,y\right) \psi _{l}\left( x\right) ,
\end{equation*}%
\begin{equation*}
S_{m}^{\left( 2\right) }\left( f;x,y\right) :=\sum\limits_{r=0}^{m-1}%
\widehat{f}\left( x,r\right) \psi _{r}\left( y\right) ,
\end{equation*}%
where 
\begin{equation*}
\widehat{f}\left( l,y\right) =\int\limits_{G_{m}}f\left( x,y\right) \psi
_{l}\left( x\right) d\mu \left( x\right)
\end{equation*}%
and 
\begin{equation*}
\widehat{f}\left( x,r\right) =\int\limits_{G_{m}}f\left( x,y\right) \psi
_{r}\left( y\right) d\mu \left( y\right) .
\end{equation*}

\section{Best Approximation}

Denote by $E_{lr}\left( f\right) $ the best approximation of a function $%
f\in C\left( G_{m}^{2}\right) $ by Vilenkin polynomials of degree $\leq l$
of a variable $x$ and of degree $\leq r$ of a variable $y$ and let $%
E_{l}^{\left( 1\right) }\left( f\right) $ be the partial best approximation
of a a function $f\in C\left( G_{m}^{2}\right) $ by Vilenkin polynomials of
degree $\leq l$ of a variable $x$, whose coefficients are continuous
functions of the remaining variable $y$. Analogously, we can define $%
E_{r}^{\left( 2\right) }\left( f\right) $.

Let $M_{L}\leq l<M_{L+1},M_{R}\leq r<M_{R+1}$ and $E_{M_{L},M_{R}}\left(
f\right) :=\left\Vert f-T_{M_{L},M_{R}}\right\Vert _{C},$ where $%
T_{M_{L},M_{R}}$ is Vilenkin polynomial of best approximation of function $f$%
. Since (see (\ref{dir})) 
\begin{equation*}
\left\Vert S_{M_{L},M_{R}}\left( f\right) \right\Vert _{C}\leq \left\Vert
f\right\Vert _{C},
\end{equation*}%
we can write 
\begin{eqnarray}
&&\left\vert S_{lr}\left( f;x,y\right) -f\left( x,y\right) \right\vert
\label{est} \\
&\leq &\left\vert S_{lr}\left( f-S_{M_{L},M_{R}}\left( f\right) ;x,y\right)
\right\vert +\left\Vert S_{M_{L},M_{R}}\left( f\right) -f\right\Vert _{C} 
\notag \\
&\leq &\left\vert S_{lr}\left( f-S_{M_{L},M_{R}}\left( f\right) ;x,y\right)
\right\vert +\left\Vert S_{M_{L},M_{R}}\left( f-T_{M_{L},M_{R}}\right)
\right\Vert _{C}  \notag \\
&&+\left\Vert f-T_{M_{L},M_{R}}f\right\Vert _{C}  \notag \\
&\leq &\left\vert S_{lr}\left( f-S_{M_{L},M_{R}}\left( f\right) ;x,y\right)
\right\vert +2E_{M_{L},M_{R}}\left( f\right) .  \notag
\end{eqnarray}

Now, we prove that the following inequality holds 
\begin{equation}
E_{M_{L},M_{R}}\left( f\right) \leq 2E_{M_{L}}^{\left( 1\right) }\left(
f\right) +2E_{M_{R}}^{\left( 2\right) }\left( f\right) .  \label{PBA}
\end{equation}%
Indeed, we have 
\begin{eqnarray}
E_{M_{L},M_{R}}\left( f\right) &\leq &\left\Vert f-S_{M_{L},M_{R}}\left(
f\right) \right\Vert _{C}=\left\Vert f-S_{M_{L}}^{\left( 1\right) }\left(
S_{M_{R}}^{\left( 2\right) }\left( f\right) \right) \right\Vert _{C}
\label{PBA1} \\
&\leq &\left\Vert f-S_{M_{L}}^{\left( 1\right) }\left( f\right) \right\Vert
_{C}+\left\Vert S_{M_{L}}^{\left( 1\right) }\left( S_{M_{R}}^{\left(
2\right) }\left( f\right) -f\right) \right\Vert _{C}  \notag \\
&\leq &\left\Vert f-S_{M_{L}}^{\left( 1\right) }\left( f\right) \right\Vert
_{C}+\left\Vert S_{M_{R}}^{\left( 2\right) }\left( f\right) -f\right\Vert
_{C}.  \notag
\end{eqnarray}%
Let $T_{M_{L}}^{\left( 1\right) }\left( x,y\right) $ be a polynomial of the
best approximation $E_{M_{L}}^{\left( 1\right) }\left( f\right) $. Then 
\begin{eqnarray}
\left\Vert S_{M_{L}}^{\left( 1\right) }\left( f\right) -f\right\Vert _{C}
&\leq &\left\Vert f-T_{M_{L}}^{\left( 1\right) }\right\Vert _{C}+\left\Vert
S_{M_{L}}^{\left( 1\right) }\left( f-T_{M_{L}}^{\left( 1\right) }\right)
\right\Vert _{C}  \label{PBA2} \\
&\leq &2\left\Vert f-T_{M_{L}}^{\left( 1\right) }\right\Vert
_{C}=2E_{M_{L}}^{\left( 1\right) }\left( f\right) .  \notag
\end{eqnarray}%
Analogously, we can prove that 
\begin{equation}
\left\Vert S_{M_{R}}^{\left( 2\right) }\left( f\right) -f\right\Vert
_{C}\leq 2E_{M_{R}}^{\left( 2\right) }\left( f\right) .  \label{PBA3}
\end{equation}%
Combining (\ref{PBA1})-(\ref{PBA3}) we obtain (\ref{PBA}).

It is easy to show that 
\begin{equation}
\left\Vert f-S_{M_{L}M_{R}}\left( f\right) \right\Vert _{C}\leq
2E_{M_{L},M_{R}}\left( f\right) .  \label{mixBA}
\end{equation}

\section{Main results}

\begin{theorem}
\label{estBA}Let $f\in C\left( G_{m}^{2}\right) $. Then the inequality 
\begin{eqnarray*}
&&\left\Vert \frac{1}{nm}\sum\limits_{l=1}^{n}\sum\limits_{r=1}^{m}\left(
e^{A\left\vert S_{lr}\left( f\right) -f\right\vert ^{1/2}}-1\right)
\right\Vert _{C} \\
&\leq &\frac{c\left( f,A\right) }{n}\sum\limits_{l=1}^{n}\sqrt{E_{l}^{\left(
1\right) }\left( f\right) }+\frac{c\left( f,A\right) }{m}\sum%
\limits_{r=1}^{m}\sqrt{E_{r}^{\left( 2\right) }\left( f\right) }
\end{eqnarray*}%
is satisfied for any $A>0$, where $c\left( f,A\right) $ is a positive
constant depend on $A$ and $f$.
\end{theorem}

We say that the function $\psi $ belongs to the class $\Psi $ if it increase
on $[0,+\infty )$ and 
\begin{equation*}
\lim\limits_{u\rightarrow 0}\psi \left( u\right) =\psi \left( 0\right) =0.
\end{equation*}

\begin{theorem}
\label{strongconv}a)Let $\varphi \in \Psi $ and let the inequality 
\begin{equation}
\overline{\lim\limits_{u\rightarrow \infty }}\frac{\varphi \left( u\right) }{%
\sqrt{u}}<\infty  \label{sufcond}
\end{equation}%
hold. Then for any function $f\in C\left( G_{m}^{2}\right) $ the equality 
\begin{equation*}
\lim\limits_{n,m\rightarrow \infty }\left\Vert \frac{1}{nm}%
\sum\limits_{l=1}^{n}\sum\limits_{r=1}^{m}\left( e^{\varphi \left(
\left\vert S_{lr}\left( f\right) -f\right\vert \right) }-1\right)
\right\Vert _{C}=0;
\end{equation*}%
is satisfied.
\end{theorem}

b) For any function $\varphi \in \Psi $ satisfying the condition 
\begin{equation}
\overline{\lim\limits_{u\rightarrow \infty }}\frac{\varphi \left( u\right) }{%
\sqrt{u}}=\infty  \label{ineq}
\end{equation}%
there exists a function $F\in C\left( G_{m}^{2}\right) $ such that 
\begin{equation*}
\overline{\lim\limits_{m\rightarrow \infty }}\frac{1}{m^{2}}%
\sum\limits_{l=1}^{m}\sum\limits_{r=1}^{m}e^{\varphi \left( \left\vert
S_{lr}\left( F;0,0\right) -F\left( 0,0\right) \right\vert \right) }=+\infty .
\end{equation*}

\section{Auxiliary Results}

\begin{lemma}
(Glukhov \cite{Gl}) Let $p\in \mathbb{N}_{+}$. Then%
\begin{equation*}
\sup\limits_{n}\left( \int\limits_{G_{m}^{p}}\frac{1}{M_{n}}\left\vert
\sum\limits_{l=M_{n}}^{M_{n+1}-1}\prod\limits_{k=1}^{p}D_{l}\left(
s_{k}\right) \right\vert d\mathbf{\mu }\left( s_{1},...,s_{p}\right) \right)
^{1/p}\leq cp.
\end{equation*}
\end{lemma}

\begin{lemma}
\label{gogoladze}(Gogoladze \cite{Go}) Let $\varphi ,\psi \in \Psi $ and the
equality 
\begin{equation*}
\lim\limits_{n,m\rightarrow \infty }\frac{1}{nm}\sum\limits_{l=1}^{n}\sum%
\limits_{r=1}^{m}\psi \left( \left\vert S_{lr}\left( f;x,y\right) -f\left(
x,y\right) \right\vert \right) =0
\end{equation*}%
be satisfied at the point $\left( x_{0},y_{0}\right) $ or uniformly on a set 
$E\subset I^{2}$. If 
\begin{equation*}
\overline{\lim\limits_{u\rightarrow \infty }}\frac{\varphi \left( u\right) }{%
\psi \left( u\right) }<\infty ,
\end{equation*}%
then the equality 
\begin{equation*}
\lim\limits_{n,m\rightarrow \infty }\frac{1}{nm}\sum\limits_{l=1}^{n}\sum%
\limits_{r=1}^{m}\varphi \left( \left\vert S_{lr}\left( f;x,y\right)
-f\left( x,y\right) \right\vert \right) =0
\end{equation*}%
is satisfied at the point $\left( x_{0},y_{0}\right) $ or uniformly on a set 
$E\subset I^{2}$.
\end{lemma}

\begin{lemma}
\label{mainlemma}Let $p>0$, $A, B \in \mathbb{N}.$ Then 
\begin{equation}
\left\{ \frac{1}{M_{A}M_{B}}\sum\limits_{n=M_{A}}^{M_{A+1}-1}\sum%
\limits_{l=M_{B}}^{M_{B+1}-1}\left\vert S_{nl}\left( f;x,y\right)
\right\vert ^{p}\right\} ^{1/p}\leq c\left\Vert f\right\Vert _{C}\left(
p+1\right) ^{2}.  \label{mainineq}
\end{equation}
\end{lemma}

\begin{proof}[Proof of Lemma \ref{mainlemma}]
Since%
\begin{eqnarray*}
&&\left\{ \frac{1}{M_{A}M_{B}}\sum\limits_{n=M_{A}}^{M_{A+1}-1}\sum%
\limits_{l=M_{B}}^{M_{B+1}-1}\left\vert S_{nl}\left( f;x,y\right)
\right\vert ^{p}\right\} ^{1/p} \\
&\leq &\left\{ \frac{1}{M_{A}M_{B}}\sum\limits_{n=M_{A}}^{M_{A+1}-1}\sum%
\limits_{l=M_{B}}^{M_{B+1}-1}\left\vert S_{nl}\left( f;x,y\right)
\right\vert ^{p+1}\right\} ^{1/\left( p+1\right) }
\end{eqnarray*}%
without lost of generality we can suppose that $p=2^{m},m\in \mathbb{N_+}$. We can
write%
\begin{equation*}
\left\vert S_{nl}\left( f;x,y\right) \right\vert ^{2}=S_{nl}\left(
f;x,y\right) \overline{S}_{nl}\left( f;x,y\right) 
\end{equation*}%
\begin{equation*}
=\int\limits_{G_{m}^{2}}f\left( x-s_{1},y-t_{1}\right) D_{n}\left(
s_{1}\right) D_{l}\left( t_{1}\right) d\mathbf{\mu }\left(
s_{1},t_{1}\right) 
\end{equation*}%
\begin{equation*}
\times \int\limits_{G_{m}^{2}}\overline{f}\left( x-s_{2},y-t_{2}\right) 
\overline{D}_{n}\left( s_{2}\right) \overline{D}_{l}\left( t_{2}\right) d%
\mathbf{\mu }\left( s_{2},t_{2}\right) 
\end{equation*}%
\begin{equation*}
=\int\limits_{G_{m}^{2}}f\left( x-s_{1},y-t_{1}\right) D_{n}\left(
s_{1}\right) D_{l}\left( t_{1}\right) d\mathbf{\mu }\left(
s_{1},t_{1}\right) 
\end{equation*}%
\begin{equation*}
\times \int\limits_{G_{m}^{2}}\overline{f}\left( x+s_{2},y+t_{2}\right)
D_{n}\left( s_{2}\right) D_{l}\left( t_{2}\right) d\mathbf{\mu }\left(
s_{2},t_{2}\right) 
\end{equation*}%
\begin{equation*}
=\int\limits_{G_{m}^{4}}f\left( x-s_{1},y-t_{1}\right) \overline{f}\left(
x+s_{2},y+t_{2}\right) 
\end{equation*}%
\begin{equation*}
\times D_{n}\left( s_{1}\right) D_{n}\left( s_{2}\right) D_{l}\left(
t_{1}\right) D_{l}\left( t_{2}\right) d\mathbf{\mu }\left(
s_{1},t_{1},s_{2},t_{2}\right) .
\end{equation*}%
Hence, we get%
\begin{equation*}
\left\vert S_{nl}\left( f;x,y\right) \right\vert ^{p}=\left( \left\vert
S_{nl}\left( f;x,y\right) \right\vert ^{2}\right) ^{p/2}
\end{equation*}%
\begin{equation*}
=\left( \int\limits_{G_{m}^{4}}f\left( x-s_{1},y-t_{1}\right) \overline{f}%
\left( x+s_{2},y+t_{2}\right) \right. 
\end{equation*}%
\begin{equation*}
\left. \times D_{n}\left( s_{1}\right) D_{n}\left( s_{2}\right) D_{l}\left(
t_{1}\right) D_{l}\left( t_{2}\right) d\mathbf{\mu }\left(
s_{1},t_{1},s_{2},t_{2}\right) \right) ^{p/2}.
\end{equation*}%
\begin{equation*}
=\int\limits_{G_{m}^{2p}}\prod\limits_{k=1}^{p/2}f\left(
x-s_{2k-1},y-t_{2k-1}\right) \prod\limits_{r=1}^{p/2}\overline{f}\left(
x+s_{2r},y+t_{2r}\right) 
\end{equation*}%
\begin{equation*}
\times \prod\limits_{i=1}^{p}D_{n}\left( s_{i}\right)
\prod\limits_{j=1}^{p}D_{l}\left( t_{j}\right) d\mathbf{\mu }\left(
s_{1},t_{1},...,s_{p},t_{p}\right) ,
\end{equation*}%
\begin{equation*}
\left\{ \frac{1}{M_{A}M_{B}}\sum\limits_{n=M_{A}}^{M_{A+1}-1}\sum%
\limits_{l=M_{B}}^{M_{B+1}-1}\left\vert S_{nl}\left( f;x,y\right)
\right\vert ^{p}\right\} ^{1/p}
\end{equation*}%
\begin{equation*}
\leq \left( \int\limits_{G_{m}^{2p}}\prod\limits_{k=1}^{p/2}\left\vert
f\left( x-s_{2k-1},y-t_{2k-1}\right) \right\vert
\prod\limits_{r=1}^{p/2}\left\vert \overline{f}\left(
x+s_{2r},y+t_{2r}\right) \right\vert \right. 
\end{equation*}%
\begin{equation*}
\left. \times \frac{1}{M_{A}M_{B}}\left\vert
\sum\limits_{n=M_{A}}^{M_{A+1}-1}\sum\limits_{l=M_{B}}^{M_{B+1}-1}\prod%
\limits_{i=1}^{p}D_{n}\left( s_{i}\right) \prod\limits_{j=1}^{p}D_{l}\left(
t_{j}\right) \right\vert d\mathbf{\mu }\left(
s_{1},t_{1},...,s_{p},t_{p}\right) \right) ^{1/p}
\end{equation*}%
\begin{equation*}
\leq \left\Vert f\right\Vert _{C}\left( \int\limits_{G_{m}^{p}}\frac{1}{M_{A}%
}\left\vert
\sum\limits_{n=M_{A}}^{M_{A+1}-1}\prod\limits_{i=1}^{p}D_{n}\left(
s_{i}\right) \right\vert d\mathbf{\mu }\left( s_{1},...,s_{p}\right) \right)
^{1/p}
\end{equation*}%
\begin{equation*}
\times \left( \int\limits_{G_{m}^{p}}\frac{1}{M_{B}}\left\vert
\sum\limits_{l=M_{B}}^{M_{B+1}-1}\prod\limits_{j=1}^{p}D_{l}\left(
t_{j}\right) \right\vert d\mathbf{\mu }\left( t_{1},...,t_{p}\right) \right)
^{1/p}
\end{equation*}%
\begin{equation*}
\leq cp^{2}\left\Vert f\right\Vert _{C}.
\end{equation*}

Lemma \ref{mainlemma} is proved.
\end{proof}

\begin{lemma}
\label{BAest}Let $f\in C\left( G_{m}^{2}\right) $ and $p>0.$ Then 
\begin{eqnarray}
&&\frac{1}{nk}\sum\limits_{l=1}^{n}\sum\limits_{r=1}^{k}\left\vert
S_{lr}\left( f;x,y\right) -f\left( x,y\right) \right\vert ^{p}  \label{BA} \\
&\leq &c^{p}\cdot \left( p+1\right) ^{2p}\left\{ \frac{1}{n}%
\sum\limits_{l=1}^{n}\left( E_{l}^{\left( 1\right) }\left( f\right) \right)
^{p}+\frac{1}{k}\sum\limits_{r=1}^{k}\left( E_{r}^{\left( 2\right) }\left(
f\right) \right) ^{p}\right\} .  \notag
\end{eqnarray}
\end{lemma}

\begin{proof}[Proof of Lemma \protect\ref{BAest}]
Since%
\begin{equation*}
\left( a+b\right) ^{\beta }\leq 2^{\beta }\left( a^{\beta }+b^{\beta
}\right) ,\beta >0
\end{equation*}%
from (\ref{est}), (\ref{PBA}), (\ref{mixBA}) and using Lemma \ref{mainlemma}
we get 
\begin{equation}
\frac{1}{M_{A}M_{B}}\sum\limits_{n=M_{A}}^{M_{A+1}-1}\sum%
\limits_{l=M_{B}}^{M_{B+1}-1}\left\vert S_{nl}\left( f;x,y\right) -f\left(
x,y\right) \right\vert ^{p}  \label{above}
\end{equation}%
\begin{equation*}
\leq \frac{2^{p}}{M_{A}M_{B}}\sum\limits_{n=M_{A}}^{M_{A+1}-1}\sum%
\limits_{l=M_{B}}^{M_{B+1}-1}\left\vert S_{nl}\left( f-S_{M_{A}M_{B}}\left(
f\right) ;x,y\right) \right\vert ^{p}
\end{equation*}%
\begin{equation*}
+\frac{2^{2p}}{M_{A}M_{B}}\left( M_{A+1}-M_{A}\right) \left(
M_{B+1}-M_{B}\right) E_{M_{A}M_{B}}^{p}\left( f\right)
\end{equation*}%
\begin{equation*}
\leq c^{p}\left( p+1\right) ^{2p}\left\Vert f-S_{M_{A}M_{B}}\left( f\right)
\right\Vert _{C}^{p}
\end{equation*}%
\begin{equation*}
+c^{p}\left( \left( E_{M_{A}}^{\left( 1\right) }\left( f\right) \right)
^{p}+\left( E_{M_{B}}^{\left( 2\right) }\left( f\right) \right) ^{p}\right)
\end{equation*}%
\begin{equation*}
\leq c^{p}\left( p+1\right) ^{2p}\left( \left( E_{M_{A}}^{\left( 1\right)
}\left( f\right) \right) ^{p}+\left( E_{M_{B}}^{\left( 2\right) }\left(
f\right) \right) ^{p}\right) .
\end{equation*}

Let $M_{L}\leq n<M_{L+1}$ and $M_{R}\leq k<M_{R+1}$. Then from (\ref{above})
we have%
\begin{equation*}
\frac{1}{nk}\sum\limits_{l=1}^{n}\sum\limits_{r=1}^{k}\left\vert
S_{lr}\left( f;x,y\right) -f\left( x,y\right) \right\vert ^{p}
\end{equation*}%
\begin{equation*}
\leq \frac{1}{nk}\sum\limits_{l=1}^{M_{L+1}-1}\sum\limits_{r=1}^{M_{R+1}-1}%
\left\vert S_{lr}\left( f;x,y\right) -f\left( x,y\right) \right\vert ^{p}
\end{equation*}%
\begin{equation*}
=\frac{1}{nk}\sum\limits_{A=0}^{L}\sum\limits_{B=0}^{R}\sum%
\limits_{l=M_{A}}^{M_{A+1}-1}\sum\limits_{r=M_{B}}^{M_{B+1}-1}\left\vert
S_{lr}\left( f;x,y\right) -f\left( x,y\right) \right\vert ^{p}
\end{equation*}%
\begin{equation*}
\leq \frac{c^{p}\left( p+1\right) ^{2p}}{nk}M_{A}M_{B}\sum\limits_{A=0}^{L}%
\sum\limits_{B=0}^{R}\left( \left( E_{M_{A}}^{\left( 1\right) }\left(
f\right) \right) ^{p}+\left( E_{M_{b}}^{\left( 2\right) }\left( f\right)
\right) ^{p}\right) 
\end{equation*}%
\begin{equation*}
\leq \frac{c^{p}\left( p+1\right) ^{2p}}{nk}\sum\limits_{A=0}^{L}\sum%
\limits_{B=0}^{R}\sum\limits_{l=M_{A-1}}^{M_{A}-1}\sum%
\limits_{r=M_{B-1}}^{M_{B}-1}\left( \left( E_{M_{A}}^{\left( 1\right)
}\left( f\right) \right) ^{p}+\left( E_{M_{B}}^{\left( 2\right) }\left(
f\right) \right) ^{p}\right) 
\end{equation*}%
\begin{equation*}
\leq \frac{c^{p}\left( p+1\right) ^{2p}}{nk}\sum\limits_{A=0}^{L}\sum%
\limits_{B=0}^{R}\sum\limits_{l=M_{A-1}}^{M_{A}-1}\sum%
\limits_{r=M_{B-1}}^{M_{B}-1}\left( \left( E_{l}^{\left( 1\right) }\left(
f\right) \right) ^{p}+\left( E_{r}^{\left( 2\right) }\left( f\right) \right)
^{p}\right) 
\end{equation*}%
\begin{equation*}
\leq \frac{c^{p}\left( p+1\right) ^{2p}}{nk}\sum\limits_{l=1}^{n}\sum%
\limits_{r=1}^{k}\left( \left( E_{l}^{\left( 1\right) }\left( f\right)
\right) ^{p}+\left( E_{r}^{\left( 2\right) }\left( f\right) \right)
^{p}\right) 
\end{equation*}%
\begin{equation*}
\leq c^{p}\cdot \left( p+1\right) ^{2p}\left\{ \frac{1}{n}%
\sum\limits_{l=1}^{n}\left( E_{l}^{\left( 1\right) }\left( f\right) \right)
^{p}+\frac{1}{k}\sum\limits_{r=1}^{k}\left( E_{r}^{\left( 2\right) }\left(
f\right) \right) ^{p}\right\} .
\end{equation*}

Lemma \ref{BAest} is proved.
\end{proof}

\section{Proofs of Main Results}

The Walsh-Paley version of Theorem \ref{estBA} were proved in \cite{GGSMH}.
Based on inequality (\ref{BA}) the same construction works for the Vilenkin
case. Therefore the proof of Theorem \ref{estBA} will be omitted.

\begin{proof}[Proof of Theorem 2]
a) It is easy to see that if $\varphi \in \Psi $, then $e^{\varphi }-1\in
\Psi .$ Besides, (\ref{sufcond}) implies the existence of a number $A$ such
that 
\begin{equation*}
\overline{\lim\limits_{u\rightarrow \infty }}\frac{e^{\varphi \left(
u\right) }-1}{e^{Au^{1/2}}-1}<\infty .
\end{equation*}%
Therefore, in view of Lemma \ref{gogoladze}, for the proof of Theorem \ref%
{strongconv} it is sufficient to prove that 
\begin{equation}
\lim\limits_{n,m\rightarrow \infty }\left\Vert \frac{1}{nm}%
\sum\limits_{l=1}^{n}\sum\limits_{r=1}^{m}\left( e^{A\left\vert S_{lr}\left(
f\right) -f\right\vert ^{1/2}}-1\right) \right\Vert _{C}=0.  \label{=0}
\end{equation}%
The validity of equality of (\ref{=0}) immediately follows from Theorem \ref%
{estBA}.

b)First of all let us prove the validity of point b) in the one-dimensional
case. In particular, prove that if $\psi \in \Psi $ and 
\begin{equation*}
\overline{\lim\limits_{u\rightarrow \infty }}\frac{\psi \left( u\right) }{{u}%
}=\infty
\end{equation*}%
there exists a function $f\in C\left( G_{m}\right) $ such that 
\begin{equation}
\overline{\lim\limits_{m\rightarrow \infty }}\frac{1}{m}\sum%
\limits_{l=1}^{m}\left( e^{\psi \left( \left\vert S_{l}\left( f;0\right)
-f\left( 0\right) \right\vert \right) }\right) =+\infty .  \label{1-dim}
\end{equation}

Let $\left\{ B_{k}:k\geq 1\right\} $ be an increasing sequence of positive
integers such that%
\begin{equation}
B_{j}>2B_{j-1},  \label{Bk>}
\end{equation}%
\begin{equation}
\frac{\psi \left( B_{j}\right) }{B_{j}}>\frac{5j\ln a}{c^{\prime }}\text{%
\thinspace \thinspace ,}  \label{>160}
\end{equation}%
where $c^{\prime }$ will be define below and $a:=\sup\limits_{j}m_{j}$.

Set 
\begin{equation*}
A_{j}:=\left[ \frac{j}{c^{\prime }}B_{j}\right] +1
\end{equation*}%
and 
\begin{equation*}
N_{A_{j}}:=\sum\limits_{k=A_{j-1}}^{A_{j}-1}\left[ \frac{m_{2k}}{2}\right]
M_{2k},
\end{equation*}%
then it is easy to see that%
\begin{equation}
N_{A_{k}}\leq a^{2A_{k}}.  \label{NAk}
\end{equation}

Set%
\begin{equation*}
f_{j}\left( x\right) :=\frac{1}{j+1}\sum\limits_{s=A_{j-1}}^{A_{j}-1}\sum%
\limits_{x_{2s+1}=0}^{m_{2s+1}-1}\cdots
\sum\limits_{x_{2A_{j}-1}=0}^{m_{2A_{j}}-1}\exp \left( -i\arg \left( 
\overline{D}_{N_{A_{j}}}\left( x\right) \right) \right)
\end{equation*}%
\begin{equation*}
\times \mathbb{I}_{I_{2A_{j}}\left(
0,...,0,x_{2s}=m_{2s}-1,x_{2s+1},...,x_{2A_{j}-1}\right) }\left( x\right) ,
\end{equation*}%
\begin{equation*}
f\left( x\right) :=\sum\limits_{j=1}^{\infty }f_{j}\left( x\right) ,f\left(
0\right) =0,
\end{equation*}%
where $\mathbb{I}_{E}$ is characteristic function of the set $E\subset G_{m}$%
.

It is easy to see that $f\in C\left( G_{m}\right) $.

We can write%
\begin{equation}
\left\vert S_{N_{A_{k}}}\left( f;0\right) -f\left( 0\right) \right\vert
=\left\vert S_{N_{A_{k}}}\left( f;0\right) \right\vert  \label{J1-J3}
\end{equation}%
\begin{equation*}
=\left\vert \int\limits_{G_{m}}f\left( t\right) \overline{D}%
_{N_{A_{k}}}\left( t\right) d\mu \left( t\right) \right\vert
\end{equation*}%
\begin{equation*}
\geq \left\vert \int\limits_{G_{m}}f_{k}\left( t\right) \overline{D}%
_{N_{A_{k}}}\left( t\right) d\mu \left( t\right) \right\vert
\end{equation*}%
\begin{equation*}
-\sum\limits_{j=k+1}^{\infty }\left\vert \int\limits_{G_{m}}f_{j}\left(
t\right) \overline{D}_{N_{A_{k}}}\left( t\right) d\mu \left( t\right)
\right\vert
\end{equation*}%
\begin{equation*}
-\sum\limits_{j=0}^{k-1}\left\vert \int\limits_{G_{m}}f_{j}\left( t\right) 
\overline{D}_{N_{A_{k}}}\left( t\right) d\mu \left( t\right) \right\vert
\end{equation*}%
\begin{equation*}
=J_{1}-J_{2}-J_{3}.
\end{equation*}

From the definition of the function $f$ we have%
\begin{equation*}
J_{1}=\frac{1}{k+1}\left\vert
\sum\limits_{s=A_{k-1}}^{A_{k}-1}\sum\limits_{t_{2s+1}=0}^{m_{2s+1}-1}\cdots
\sum\limits_{t_{2A_{k}-1}=0}^{m_{2A_{k}}-1}\right.
\end{equation*}%
\begin{equation*}
\left. \int\limits_{_{I_{2A_{k}}\left(
0,...,0,t_{2s}=m_{2s}-1,t_{2s+1},...,t_{2A_{k}-1}\right) }}\exp \left(
-i\arg \left( \overline{D}_{N_{A_{k}}}\left( t\right) \right) \right) 
\overline{D}_{N_{A_{k}}}\left( t\right) d\mu \left( t\right) \right\vert
\end{equation*}%
\begin{equation*}
=\frac{1}{k+1}\sum\limits_{s=A_{k-1}}^{A_{k}-1}\sum%
\limits_{t_{2s+1}=0}^{m_{2s+1}-1}\cdots
\sum\limits_{t_{2A_{k}-1}=0}^{m_{2A_{k}}-1}
\end{equation*}%
\begin{equation*}
\int\limits_{I_{2A_{k}}\left(
0,...,0,t_{2s}=m_{2s}-1,t_{2s+1},...,t_{2A_{k}-1}\right) }\left\vert
D_{N_{A_{k}}}\left( t\right) \right\vert d\mu \left( t\right) .
\end{equation*}%
Since (see \cite{GMJ} )%
\begin{equation*}
\left\vert D_{N_{A_{k}}}\left( t\right) \right\vert \geq cM_{2s}
\end{equation*}%
for 
\begin{equation*}
t\in I_{2s+1}\left( 0,...,0,t_{2s}=m_{2s}-1\right) ,s=A_{k-1},...,A_{k}-1
\end{equation*}%
from (\ref{Bk>}) we cam write%
\begin{equation}
J_{1}\geq \frac{c}{k}\sum\limits_{s=A_{k-1}}^{A_{k}-1}M_{2s}\sum%
\limits_{t_{2s+1}=0}^{m_{2s+1}-1}\cdots
\sum\limits_{t_{2A_{k}-1}=0}^{m_{2A_{k}}-1}\mu \left( I_{2A_{k}}\left(
0,...,0,t_{2s}=m_{2s}-1,t_{2s+1},...,t_{2A_{k}-1}\right) \right)  \label{J1}
\end{equation}%
\begin{equation*}
\geq \frac{c}{k}\left( A_{k}-A_{k-1}\right) \geq \frac{c_{0}A_{k}}{k}.
\end{equation*}

For $J_{2}$ we have%
\begin{equation}
J_{2}\leq \sum\limits_{j=k+1}^{\infty }\frac{1}{j+1}\sum%
\limits_{s=A_{j-1}}^{A_{j}-1}\frac{1}{M_{2s}}N_{A_{k}}  \label{J2}
\end{equation}%
\begin{equation*}
\leq \frac{c}{k}\sum\limits_{s=A_{k}}^{\infty }\frac{1}{M_{2s}}N_{A_{k}}\leq 
\frac{c}{k}.
\end{equation*}

By (\ref{dir2}) and from the construction of the function $f_{j}$ we can
write 
\begin{equation*}
\text{supp}\left( f_{j}\right) \cap \text{supp}\left( D_{N_{A_{k}}}\right)
=\varnothing ,j=1,2,...,k-1
\end{equation*}%
consequently%
\begin{equation}
J_{3}=0.  \label{J3}
\end{equation}

Combining (\ref{>160}), ( \ref{J1-J3})-(\ref{J3}) we conclude that%
\begin{equation*}
\left\vert S_{N_{A_{k}}}\left( f;0\right) \right\vert =\left\vert
S_{N_{A_{k}}}\left( f;0\right) -f\left( 0\right) \right\vert \geq \frac{%
c^{\prime }A_{k}}{k}\geq B_{k},
\end{equation*}%
\begin{equation*}
\psi \left( \left\vert S_{N_{A_{k}}}\left( f;0\right) \right\vert \right)
\geq \psi \left( B_{k}\right) \geq \frac{5k\ln a}{c^{\prime }}B_{k}\geq
5\left( A_{k}-1\right) \ln a.
\end{equation*}

Write $\varphi \left( u\right) =\lambda \left( u\right) \sqrt{u}$ and define 
$\psi \left( u\right) :=\lambda \left( u^{2}\right) u.$ Then%
\begin{equation*}
\overline{\lim\limits_{u\rightarrow \infty }}\frac{\psi \left( u\right) }{u}=+\infty .
\end{equation*}%
Therefore there exist a function $f\in C\left( G_{m}\right) $ for which%
\begin{equation}
\psi \left( \left\vert S_{N_{A_{k}}}\left( f,0\right) \right\vert \right)
>5\left( A_{k}-1\right) \ln a.  \label{ineq2}
\end{equation}%
Set%
\begin{equation*}
F\left( x,y\right) :=f\left( x\right) f\left( y\right) .
\end{equation*}%
It is easy to show that%
\begin{eqnarray*}
&&\varphi \left( \left\vert S_{N_{A_{k}},N_{A_{k}}}\left( F;0,0\right)
\right\vert \right) \\
&=&\varphi \left( \left\vert S_{N_{A_{k}}}\left( f;0\right) \right\vert
^{2}\right) \\
&=&\lambda \left( \left\vert S_{N_{A_{k}}}\left( f;0\right) \right\vert
^{2}\right) \left\vert S_{N_{A_{k}}}\left( f;0\right) \right\vert \\
&=&\psi \left( \left\vert S_{N_{A_{k}}}\left( f;0\right) \right\vert \right)
.
\end{eqnarray*}%
Consequently, from (\ref{ineq2}) and (\ref{NAk}) we have%
\begin{eqnarray*}
&&\frac{1}{N_{A_{k}}^{2}}\sum\limits_{i=1}^{N_{A_{k}}}\sum%
\limits_{j=1}^{N_{A_{k}}}e^{\varphi \left( \left\vert S_{ij}\left(
F;0,0\right) \right\vert \right) } \\
&\geq &\frac{1}{N_{A_{k}}^{2}}e^{\varphi \left( \left\vert
S_{N_{A_{k}},N_{A_{k}}}\left( F;0,0\right) \right\vert \right) } \\
&=&\frac{1}{N_{A_{k}}^{2}}e^{\psi \left( \left\vert S_{N_{A_{k}}}\left(
f;0\right) \right\vert \right) } \\
&\geq &\frac{e^{5\left( A_{k}-1\right) \ln a}}{a^{4A_{k}}}\rightarrow \infty 
\text{ as~\ }k\rightarrow \infty .
\end{eqnarray*}%
Theorem \ref{strongconv} is proved.
\end{proof}

\end{document}